\documentclass[a4paper,12pt]{article}

\usepackage[utf8]{inputenc}
\usepackage[T1]{fontenc}
\usepackage[english]{babel}
\usepackage{url}

\usepackage[a4paper,margin=1in]{geometry}

\usepackage{amsmath,amssymb,amsfonts,amsthm}
\usepackage{mathrsfs,dsfont}
\usepackage{mathtools}

\usepackage{pgf}
\usepackage{graphicx}

\usepackage{xcolor}
\usepackage[german=guillemets]{csquotes} 

\usepackage{authblk}

\usepackage{hyperref}

\DeclareFontFamily{U}{mathx}{}
\DeclareFontShape{U}{mathx}{m}{n}{<->mathx10}{}
\DeclareSymbolFont{mathx}{U}{mathx}{m}{n}
\DeclareMathAccent{\widehat}{0}{mathx}{"70}
\DeclareMathAccent{\widecheck}{0}{mathx}{"71}

\theoremstyle{plain}
\newtheorem{theorem}{Theorem}[section]
\newtheorem{lemma}[theorem]{Lemma}
\newtheorem{proposition}[theorem]{Proposition}

\theoremstyle{definition}
\newtheorem{remark}[theorem]{Remark}

\newtheorem*{acknowledgement}{Acknowledgement}

\numberwithin{equation}{section}

\newcommand{\nat}{\mathds{N}}
\newcommand{\real}{\mathds{R}}
\newcommand{\complex}{\mathds{C}}
\newcommand{\rd}{{\mathds{R}^d}}
\newcommand{\Ee}{\mathds{E}}
\newcommand{\Pp}{\mathds{P}}
\newcommand{\Ss}{\mathds{S}}
\newcommand{\One}{\mathds{1}}

\newcommand{\Leb}{\operatorname{Leb}}
\newcommand{\Law}{\operatorname{Law}}
\newcommand{\diver}{\operatorname{div}}

\begin{document}
	\title{\bfseries Invariant Measures of L\'evy-driven Stochastic Differential Equations}
	
	\author[a]{V.\ Knopova}
	\author[b]{Y.\ Mokanu}
	\author[c]{R.L.\ Schilling}
	
	\renewcommand\Affilfont{\itshape\footnotesize}
	\affil[a]{\footnotesize Institute of Mathematics, NAS of Ukraine, Tereschenkivska st.~3, 01024,  Kyiv, Ukraine.  \normalfont\url{vicknopova@gmail.com}}
	
	\affil[b]{Department of Mathematical Analysis and Probability Theory, National Technical University of Ukraine \enquote{Igor Sikorsky Kyiv Polytechnic Institute}, 37 Beresteyskyi avenue, 03056 Kyiv, Ukraine. \normalfont\url{yana.mokanu3075@gmail.com}}
	
	\affil[c]{Fakultät Mathematik, Technische Universität Dresden, 01062 Dresden, Germany. \normalfont\url{rene.schilling@tu-dresden.de}}
	
	\date{\normalsize To appear in: \emph{Tbilisi Analysis and PDE Seminar: Dedicated to the 80th Anniversary of Roland Duduchava}. Springer 2026.}
	
	\maketitle 
	
\begin{abstract} \noindent
	We study the structure and regularity of (infinitesimally) invariant measures of the solutions to stochastic differential equations $dX_t = b(X_t)\,dt + dZ_t$, where $(Z_t)_{t\geq 0}$ is a L\'evy process. We show, in particular, that the invariant measure has to satisfy a Volterra-type convolution equation; since we can obtain the kernels explicitly, we are able to apply regularity methods from harmonic analysis. As an application, we get a very short proof -- in any dimension -- of a classic result due to Sato and Yamazato on the form of the invariant measure of a generalized Ornstein--Uhlenbeck process.
	
	
\end{abstract}

\section{Introduction}
In this paper we study the structure and regularity of (infinitesimally) invariant measures of the solutions of a class of L\'evy-driven stochastic differential equations (SDEs). Since we deal with L\'evy processes, we will mainly use methods from Fourier analysis and partial differential equations with constant coefficients. One of the earliest papers in this direction is the contribution by Albeverio, Rüdiger and Wu \cite{ARJ00}, where SDEs with linear drift and $\alpha$-stable noise are investigated. Recently, this line of research was taken up by Behme and Oechsler \cite{BO24,Ol24} for a rather general class of L\'evy-type operators but -- for technical reasons -- with an emphasis on the one-dimensional case: The main contribution is to transform the equation of infinitesimal invariance $L^*\rho =0$ (in the space of distributions $\mathcal{D}'$ where $L$ is the infinitesimal generator of the underlying process) into a Volterra-type equation for the invariant density $\rho$. Since $L$ may be a rather general pseudo-differential operator, the results are somewhat involved and less explicit. 

Our aim in this paper is different: We focus on generators $L$ of the form \enquote{L\'evy generator plus variable drift} (being the natural generalization of \cite{ARJ00}), which allows for a much more straightforward  derivation of an explicit Volterra equation in any dimension. Moreover, we make use of this equation to better understand the structure of the (infinitesimally) invariant density $\rho$ and to derive smoothness or integrability properties of $\rho$. For dimensions $d>1$ such results seem to be new in connection with non-local operators, and a considerable part of our contribution deals with the  breakdown  of the \emph{fundamental theorem of integral calculus} in higher dimensions, since for $F\in C^1(\rd,\rd)$ the condition $\diver F=0$ does, in general, not imply that $F$ is a constant. The other innovation is to introduce techniques from Calder\'on--Zygmund operators to the study of invariant measures. We do not strive for optimal results, we rather understand our contribution as a first non-trivial attempt to deal with higher dimensions, which should lead to subsequent developments.

As an application of our findings, we obtain a very short proof of a classic result due to Sato \& Yamazato  \cite{SY84} and Masuda  \cite{Ma04} on the form of the invariant measure of a generalized Ornstein--Uhlenbeck process.

\medskip
Our paper is organized as follows: In Section~\ref{intro} we introduce L\'evy processes and convolution semigroups, their generators and (infinitesimally) invariant measures. Section~\ref{volt} contains the basic assumptions \textbf{(A1)}--\textbf{(A3)} and the main results, Theorem~\ref{volt-05} and~\ref{volt-09}, along with a Helmholtz-type decomposition result for $d$-dimensional vector fields (Lemma~\ref{volt-04}, see also Lemma~\ref{proo-02}), which is interesting in its own right. The proofs of the main theorems are postponed to Section~\ref{proo}. In Section~\ref{exa} we apply our technique to the classical L\'evy-driven Ornstein--Uhlenbeck process with values in $\rd$, and we recover the classical form of the invariant measure due to Sato \& Yamazato and Masuda with a very short argument. In the last section (Section~\ref{app}) we discuss the convergence speed to equilibrium of two models (a L\'evy-driven SDE and a reflecting Brownian motion), which have the same invariant measure $\lambda e^{-\lambda x}\,dx$. Depending on the size of the parameter $\lambda$, the L\'evy model outperforms the Brownian motion model or \emph{vice versa}.

\medskip\noindent
\textbf{Notation.}
	Most of our notation should be standard or self-explanatory. Unless indicated otherwise, all spaces of integrable functions $L_p, L_{p,\infty}$ etc.\ are on $\rd$ and relative to Lebesgue measure $dx$. We will frequently use the convention from ergodic theory and write $\mu(f):=\int f\,d\mu$. By $C^k(\rd)$ [and $C^{k,\alpha}(\rd)$] we denote the $k$ times continuously differentiable functions [whose $k$-th derivatives are Hölder continuous of order $\alpha\in (0,1]$] and $\mathcal{S}(\rd)$ are the rapidly decreasing smooth functions with dual $\mathcal{S}'(\rd)$.  We set $\Lambda(x) := (1+|x|^2)^{1/2}$.  The Fourier transform and its inverse are
	\begin{gather*}
		\smash[t]{\mathscr{F} f(\xi) = \widehat f(\xi) = \frac 1{(2\pi)^d}\int_{\rd} f(x) e^{-i\xi\cdot x}\,dx,
		\quad 
		\mathscr{F}^{-1} g(x) = \widecheck g(x) = \int_{\rd} g(\xi) e^{i \xi\cdot x}\,d\xi.}
	\end{gather*}
	We use $f\lesssim g$ and $f\asymp g$ as a shorthand for $f(x)\leq C g(x)$ and $c g(x)\leq f(x)\leq C g(x)$ for suitable constants $0<c\leq C<\infty$,  and $\int_a^b := \int_{(a,b]}$. Finally, $\omega_d = \mathrm{Leb}(B_1(0)) = \pi^{d/2}/{\Gamma\left(\frac d2+1\right)}$ is the volume of the unit ball, hence $d\cdot\omega_d = 2\pi^{d/2}/\Gamma\left(\frac d2\right)$ is  the  volume of  the  unit sphere $\Ss^{d-1}\subset\rd$.

\section{Preliminaries}\label{intro}

\noindent
Let $X=(X_t)_{t\geq 0}$ be a L\'evy process taking values in $\rd$, i.e.\ a stochastic process with stationary and independent increments and c\`adl\`ag paths $t\mapsto X_t$. We can characterize the law of $X$ via the characteristic function,
\begin{gather*}
	\Ee e^{i\xi\cdot X_t} = e^{-t\chi(\xi)}
\end{gather*}
where the exponent $\chi:\rd\to\complex$ is given by the \emph{L\'evy--Khintchine} formula:
\begin{gather}\label{intro-e04}
	\chi(\xi)
	= -ia\cdot\xi + \frac 12\xi\cdot Q\xi 
	+  \int_{y\neq 0} \left(1-e^{i\xi\cdot y}+i\xi\cdot y\One_{(0,1)}(|y|)\right) \nu(dy).
\end{gather}
The L\'evy triplet $(a,Q,\nu)$, comprising $a\in\rd$, a positive semidefinite matrix $Q\in\real^{d\times d}$ and a Radon measure $\nu$ on $\rd\setminus\{0\}$ such that $\int_{y\neq 0} 1\wedge|y|^2\,\nu(dy)<\infty$, uniquely determines $\chi$. To keep notation simple,  we will assume $a=0$ and $Q=0$.

The characteristic exponent $\chi$ is also the symbol of the infinitesimal generator of the (transition semigroup induced by the) L\'evy process,
\begin{gather}\label{intro-e07}
	-\chi(D)f 
	= \mathscr{F}^{-1}(-\chi\mathscr{F} f),\quad f\in C_c^\infty(\rd).
\end{gather}
If the measure $\nu$ satisfies the integrability condition $\int_{|y|>1} |y|\,\nu(dy) < \infty$, then we have $\chi(\xi) = -i\ell\xi + \psi(\xi)$, where 
\begin{gather}\label{intro-e09}
	\ell := \int_{|y|>1} y\, \nu(dy)
	\quad\text{and}\quad
	\psi(\xi) := \int_{y\neq 0} \left(1-e^{i\xi\cdot y} + i\xi\cdot y\right) \nu(dy).
\end{gather} 
Inverting the Fourier transform  shows that $\psi(D)$ can also be written as
\begin{gather}\label{intro-e10}
	-\psi(D)f(x) = \int_{y\neq 0} \left(f(x+y) -f(x) - \nabla f (x) \cdot y\right) \nu(dy), \quad f\in C_b^2(\rd).
\end{gather}
The formal adjoint of $\psi(D)$ is the operator with symbol $\widetilde\psi(\xi):=\psi(-\xi)$.  We are interested in the operator
\begin{gather}\label{intro-e12}
	Lf = b(x)\cdot\nabla f  - \chi(D)f,\quad f\in C_b^2(\rd),
\end{gather}
with variable drift $b: \rd\to \rd$. If $(L,C_c^2(\rd))$ extends to the generator of a Markov process, say $(X_t)_{t\geq 0}$, then this process is the solution to the following stochastic differential equation
\begin{gather}\label{intro-e14}
	dX_t = b(X_t)\,dt + dZ_t, \quad X_0= x, 
\end{gather}
where $(Z_t)_{t\geq 0}$ is a L\'evy  process whose characteristic exponent is $\chi(\xi)$. The converse assertion is also true, cf.\ \cite{kur11}: If \eqref{intro-e14} has a unique (strong or weak) solution, then this solution is a Markov process with semigroup $T_tf(x) = \Ee^x f(X_t)$, $f\in C_b(\rd)$, and generator $L$,  such that the martingale problem for $(L,C_b^2(\rd))$ is well-posed. 

We want to study the properties of probability measures, which are invariant under the law of $X_t$.  Recall that a finite  measure $\pi$  on $(\rd, \mathscr{B})$ is said to be \emph{invariant}, if the formal adjoint semigroup satisfies $T_t^* \pi  = \pi$,  i.e.\ for any Borel set $B\in \mathscr{B}$  we have 
\begin{gather}\label{intro-e16}
	\smash[t]{\pi(T_t\One_B) := \int_\rd  T_t \One_B(x)\,\pi(dx) = \pi (B)}
	\quad\text{for all $t>0$}.
\end{gather}
If $\pi$ is an invariant measure,  then for any $f\in C_c^\infty(\rd)$ we have by Dynkin's formula and Fubini's theorem 
\begin{align}
	0 
	= \pi(T_t f) - \pi(f)
	&= \pi \left(\int_0^t L T_s f \,ds\right) =  \pi\left(L\int_0^t T_sf\,ds\right)\label{intro-e17a} \\
	&= \pi \left(\int_0^t T_s L f \,ds\right) = t \pi (Lf).\label{intro-e17b}
\end{align}
Thus, the invariance of a finite measure implies, via \eqref{intro-e17b}, the so-called \emph{infinitesimal invariance}: 
\begin{gather}\label{intro-e18}
	\pi(Lf) = \int_\rd  Lf(x)\, \pi(dx) = 0 \quad \text{for all\ \ } f\in C_c^\infty(\rd). 
\end{gather} 
From \eqref{intro-e17a} we see, that  an  infinitesimally invariant measure is also invariant, if $C_c^\infty(\rd)$ is an operator core,  since then  we  can  find a sequence $(f_n)_{n\in\nat}\subset C_c^\infty(\rd)$ such that $\left(f_n;\: Lf_n\right) \to \left(\int_0^t T_s f\,ds;\: L\int_0^t T_s f\,ds\right)$ strongly.

We assume also that $\pi$ is absolutely continuous with respect to Lebesgue measure and we denote its density by $\rho(x)$. In this case \eqref{intro-e18} is equivalent to 
\begin{gather}\label{intro-e20}
	L^* \rho = 0
\end{gather} 
in the sense of Schwartz distributions; $L^*$ is the formal adjoint of $L$. In Theorem~\ref{volt-05} we show that under certain assumptions \eqref{intro-e20} is equivalent to 
\begin{gather}\label{intro-e22}
	b_\ell (x)\rho(x) := \left(b(x) -\ell\right)\rho(x) = \nabla K * \rho(x), 
\end{gather}
where the kernel $K$ is explicitly given. In the multi-dimensional case the kernel $\nabla K$  can be expressed in terms of the Calder\'on--Zygmund operator $\nabla^2 P$, where $P$ is the Newton kernel. This allows us to use Calder\'on--Zygmund estimates and to get integrability properties of $\rho$. Calder\'on--Zygmund estimates  play a central role in regularity theory for (integro-)differential equations, see e.g.\ \cite{Gr08}, \cite{St70} and \cite{St93}. Theorem~\ref{volt-05} is not new (in the sense that the Volterra-type equation for $\rho$ could be derived from \cite{BO24}), but our proof is much more straightforward and reveals the structure of the convolution kernel $\nabla K$. This, in turn allows us to obtain further properties of $\rho$ via Calder\'on--Zygmund estimates, cf.\ Theorem~\ref{volt-09}.

The problem  of finding the (infinitesimally) invariant measure  is closely related to the Langevin dynamics. One can use \eqref{intro-e14} for sampling from the invariant distribution $\pi$. If we know that $\pi(dx) = \rho(x)\,dx$, we can use \eqref{intro-e22} in order to construct suitable $\nu(\cdot)$ and $b(x)$, write the equation \eqref{intro-e14} and sample from  $\pi$  using the Euler scheme for \eqref{intro-e14} and sufficiently large $t\gg 1$. Alternatively, one may consider a (reflecting) diffusion as driving noise for the SDE. We provide a simple numerical example, in which the same invariant measure $\rho(x) = \lambda e^{-\lambda x}$, $x>0$, is simulated by  both \eqref{intro-e14} and the  reflecting  diffusion 
\begin{gather}\label{intro-e26}
	dY_t = -\mu dt + dB_t + dL_t^0, \quad  t\geq 0, 
\end{gather}
where $\mu =\frac{\lambda}{2}$,   and $(L_t^0)_{t\geq 0}$ is the local time  of $(Y_t)_{t\geq 0}$ at $0$. In this case, the speed of convergence  depends on the drift:  In some cases the  L\'evy-driven models give faster $L_1$-convergence, and in some cases the dynamics, given by the SDE with driving noise \eqref{intro-e26} is faster.

\section{A Volterra-type equation for the invariant measure}\label{volt}

\noindent
We can write the operators $\nabla$, $\Delta$ and $(-\Delta)^{\alpha}$, $\alpha\in\real\setminus\{0\}$, as pseudo-differential operators \eqref{intro-e07} with the symbols $i\xi$, $-|\xi|^2$ and $|\xi|^\alpha$, respectively. By  $P(x)$ we denote  the  Newton potential, in $\rd$, $d\geq 2$, which is the inverse Fourier transform of $|\xi|^{-2}$ in the sense of distributions.  We have   
\begin{gather*}
	P(x)
	=
	\begin{cases}
		\displaystyle C_d |x|^{2-d}, &\text{for\ \ } d\geq 3,\\
		\displaystyle C_d\log |x|, &\text{for\ \ }  d=2,
	\end{cases}
	\qquad\text{with}\qquad
	C_d = 
	\begin{cases}\frac 1{d(d-2)\omega_d}, & d\geq 3, \\
		-\frac{1}{2\pi}, &  d=2. 
	\end{cases}
\end{gather*}
Throughout the paper we will use the following assumptions:
\begin{description}
	\item[\textup{(A1)}]
	The SDE \eqref{intro-e14} has a unique weak solution $(X_t)_{t\geq 0}$, which has an  invariant measure,   and $C_c^\infty(\rd)$ is a core of the generator $(L,\mathcal{D})$.

	\item[\textup{(A2)}]
	The drift coefficient $b(x)$ is measurable and polynomially bounded:\\ For some $N\in\nat$ one has $|b(x)|\lesssim \Lambda^N(x)$.

	\item[\textup{(A3)}]  $\int_{|y|>1} |y|\,\nu(dy)<\infty$. 
\end{description}
Let us briefly discuss these assumptions

\begin{remark}\label{volt-03}
	\emph{a.}\ \ 
	The fact that $C_c^\infty(\rd)$ is an operator core ensures that the notions of \emph{invariant measure} and \emph{infinitesimally invariant measure} coincide. A sufficient criterion  for the existence of an invariant measure  in (A1)  is the Krylov--Bogoliubov theorem, see e.g.\ \cite[Ch.~3.1]{daP-Zab96}: If the solution to \eqref{intro-e14} has the $C_b$-Feller property, i.e.\ $x\mapsto T_t f(x):= \Ee^x f(X_t)$ is for every $f\in C_b(\rd)$ a continuous function and if (the resolvent of) the probability kernels $p_t(x,B):=\Pp^x(X_t\in B)$ is tight, then there is an  invariant measure.   A sufficient condition for the uniqueness of the  invariant measure  is irreducibility combined with the strong Feller property, cf.\ \cite[Ch.~4.2]{daP-Zab96}. 
	
	A sufficient condition for the Feller property is $b\in C(\rd)$ and $\lim_{|x|\to\infty} \frac{b(x)}{|x|}=0$, see e.g.\ \cite[Ex.~4.1]{Ku18a}. This paper also shows that the symbol of the generator $L$ is of the form $q(x,\xi)=-ib(x)\xi + \psi(\xi)$. The required tightness follows from the maximal inequality for Feller and L\'evy-type processes, cf.\ \cite[Cor.~5.2]{BSW13}
	\begin{gather*}
		\Pp^x\left(\sup\nolimits_{s\leq t}|X_t-x|>r\right)
		\leq c t \sup_{|y-x|\leq r} \sup_{|\xi|\leq 1/r} |q(y,\xi)|,\quad x\in\rd,\; r,t>0,
	\end{gather*}
	if $\lim_{|x|\to\infty} b(x)/|x|=0$.  Sufficient (as well as necessary) conditions for the strong Feller property are  discussed  in \cite{SW12}. 
	
	Further criteria for the ergodicity of L\'evy-driven SDEs can be found in \cite{K09} or \cite{XZ20}. Since the weak uniqueness of the SDE and the well-posedness of the martingale problem are equivalent, see \cite{kur11}, we could equivalently require that the $(C_c^\infty(\rd),L)$ martingale problem is well-posed. In connection with martingale problems, the Feller property is discussed in \cite{Ku18b}. 
	
	\medskip\noindent\emph{b.}\ \ 
	(A3) guarantees that $\psi(\xi)\lesssim |\xi|$ for all $|\xi|\leq 1$. Indeed, distinguishing between $|y\xi|\leq 1$ and $|y\xi|>1$, and then restricting ourselves to $|\xi|\leq 1$, shows that
	\begin{gather*}
		\left|1-e^{iy\xi}-iy\xi \right| 
		\lesssim |y\xi|^2\wedge |y\xi| 
		\leq (|y|^2\wedge|y|)|\xi|.
	\end{gather*}
	Thus, because of \eqref{intro-e04},
	\begin{gather*}
		|\psi(\xi)| \lesssim \left[\int_{|y|\leq 1}|y|^2\,\nu(dy) + \int_{|y|> 1} |y|\,\nu(dy)\right] |\xi|,\quad |\xi|\leq 1.
	\end{gather*}
	
	From the L\'evy--Khintchine representation \eqref{intro-e04} we see that $|\psi(\xi)|\lesssim (1+|\xi|^2)$ holds for all $\xi\in\rd$, thus we have under (A3) that $|\psi(\xi)|\lesssim \max\{|\xi|,|\xi|^2\}$, $\xi\in\rd$.
\end{remark}

In $\real^3$ the following Helmholtz decomposition of a vector field $F\in C^1(\real^3,\real^3)$ is known: $F = F^{\mathsf{irr}} + F^{\mathsf{sol}}$ where $F^{\mathsf{irr}}=\nabla\varphi$ comes from a potential and $F^{\mathsf{sol}}$ is divergence free, i.e.\ $\diver F^{\mathsf{sol}}=0$. For $d=1$ this is the fundamental theorem of integral calculus: $\varphi = F'$ and $F^{\mathsf{sol}} = \text{const}$; for $d\geq 2$ there exist various generalizations, see e.g.\ \cite{GloeRich23}. We use the following observation: 
\begin{lemma}\label{volt-04}
	Let $d\geq 2$ and $P$ the Newton potential. If $F\in C^1(\rd,\rd)$ is such that $\diver F$ is Hölder continuous for some $\alpha\in (0,1]$ and
	\begin{itemize}
		\item $\diver F\in L_1$ for $d\geq 3$,
		\item $\Lambda^\epsilon \diver F\in L_1$ for some $\epsilon>0$ for $d=2$,
	\end{itemize}
	then $F^{\mathsf{sol}}:= F - \nabla P*\diver F$ is divergence-free and we define 
	\begin{gather}\label{volt-e04}
		F^{\mathsf{irr}}:=\nabla P*\diver F 
		\quad\text{and}\quad
		F^{\mathsf{sol}}:= F-\nabla P*\diver F.
	\end{gather}
\end{lemma}
\begin{proof}
	We treat only $d=2$, as $d>2$ is similar setting $\epsilon = 0$. Take a test function $\phi\in C_c^\infty(\rd)$ such that $\One_{B_1(0)}\leq \phi \leq \One_{B_2(0)}$. Note that $P*\diver F$ exists and is finite, since by Young's inequality  and Peetre's inequality 
	\begin{gather*}
		P*\diver F(x) 
		\lesssim
		\left\|\phi P\right\|_{L_1} \sup_{y\in B_2(x)}|\diver F(y)|
		+ \Lambda^\epsilon(x)\left\|(1-\phi)\Lambda^{-\epsilon}P\right\|_{L_\infty} \|\Lambda^\epsilon\diver F\|_{L_1}.
	\end{gather*}
	By a standard result from the theory of the Poisson equation with H\"older continuous force, see \cite[Lem.~4.2, p.~55]{GT98},  we have $P*\diver F\in C^2$ and $\Delta(P*\diver F)=\diver F$. This finishes the proof.
\end{proof}

We can now state our main results. The proofs are postponed to the next section. We write $K(x) := (2\pi)^{-d}\,  \widetilde{\psi}(D) P(x)$, if $d\geq 2$, and $\nabla K(x) := - \overline{\nu}(x)$, if $d=1$,  where 
\begin{gather*}
	\overline{\nu}(x)= \nu(x,\infty)\One_{(0,\infty)}(x) - \nu(-\infty, x)\One_{(-\infty,0)}(x)
\end{gather*} 
is the tail function of the L\'evy measure. As usual, $W^{k,p}$ is the Sobolev space of order $k\in\nat$ in $L_p$, i.e.\ $W^{k,p} := \left\{ f\in L_p \mid \nabla^k f\in L_p\right\}$.
\begin{theorem}\label{volt-05}
	Assume \textup{(A1)}, \textup{(A2)} and  \textup{(A3)}.  The equation $L^*\rho = 0$ in $\mathcal{S}'(\rd)$ is equivalent to
	\begin{gather}\label{volt-e06}
		\diver\left(b_\ell\rho\right) = \diver\left(\nabla\widetilde{\psi}(D)(-\Delta)^{-1}\rho\right)
		\quad\text{in}\quad \mathcal{S}'(\rd).
	\end{gather}
	Moreover, 
	\begin{enumerate} 
		\item
		if $\rho\in L_1$ is a solution in $\mathcal{S}'(\rd)$ either to
		\begin{gather}\label{volt-e06a}
			(b_\ell\rho)^{\mathsf{irr}} = \nabla K * \rho,
			\intertext{or to}\label{volt-e06b}
			(b_\ell\rho) = \nabla K * \rho,
		\end{gather}
		then the measure $\pi(dx) = \rho(x)\,dx$ satisfies \eqref{intro-e18}, i.e.\ it is \textup{(}infinitesimally\textup{)} invariant for the operator $L$.
		
		\item
		if the measure $\pi(dx) = \rho(x)\,dx$ satisfies \eqref{intro-e18}, and if one of the following conditions hold
		\begin{description}
			\item[i)] $d=1$;
			\item[ii)]
			$d=2$,   $\rho \in W^{1,1}\cap W^{1,\infty}$  and  $ \rho b_\ell  \in C^{1,\alpha}(\rd,\rd)$  for some $\alpha\in (0,1]$,  $\diver(\rho b_\ell)\in L_\infty$ and $\Lambda^\epsilon\diver(\rho b_\ell)\in L_1$ for some $\epsilon>0$;
			\item[iii)]
			$d\geq 3$,  $\rho \in W^{1,1}\cap W^{1,\infty}$,  $\rho b_\ell  \in C^{ 1,\alpha}(\rd,\rd)$ for some $\alpha\in (0,1]$ and $\diver(\rho b_\ell)\in L_1 \cap L_\infty$;  
			\item[iv)]  
			$d \geq 3$, $\rho\in L_2$, $\rho b_\ell \in  C^{1,\alpha}(\rd,\rd)\cap L_1\cap L_2$  and $\diver(\rho b_\ell)\in L_1$;
		\end{description}  
		then $\rho$ satisfies \eqref{volt-e06a} in $\mathcal{S}'(\rd)$. 
	\end{enumerate}
\end{theorem}
Our second result concerns the integrability properties of $\rho$. By definition, $\rho$ is a probability density, so it is trivially in $L_1$. But we get better properties from the fact that $\rho$ is the solution to \eqref{volt-e06a}, resp., \eqref{volt-e06b}.  

We discuss first the one-dimensional case. We need to assume that the tail function $\overline\nu(x)$ is in $L_1$. Using Tonelli's theorem we see that this is equivalent to
\begin{align*}
	\int_{-\infty}^\infty |\overline{\nu}(x)|\,dx
	&= \int_0^\infty \int_x^\infty \nu(dy)\,dx + \int_{-\infty}^0\int_{-\infty}^x \nu(dy)\,dx\\
	&= \int_0^\infty \int_0^y dx\, \nu(dy) + \int_{-\infty}^0\int_{y}^0 dx\, \nu(dy)
	= \int_{-\infty}^\infty |y|\,\nu(dy) < \infty.
\end{align*}
\begin{proposition}\label{volt-07}
	Let $d=1$, assume \textup{(A1)}--\textup{(A3)} and $\int_{|y|\leq 1} |y|\,\nu(dy)<\infty$. Suppose that  \eqref{volt-e06b} is satisfied for some $\rho\in L_1$. Then $b\rho \in L_1$. 
\end{proposition}
This proposition simply follows from  the fact that  $\overline\nu\in L_1$ and Young's inequality: The convolution operator  $T f := \nabla K *f  = -\overline\nu*f$  is continuous in $L_1$. Note that $b$ is a pointwise multiplier in $L_1$ if, and only if, $b\in L_\infty$; but $b$ is, in general, unbounded. In order to have $b\rho\in L_1$ we need more regularity of $\rho$. Proposition~\ref{volt-07} means that, in the one-dimensional case, this regularity must be due to  $\overline\nu\in L_1$ \ --- this condition ensures that the right-hand side of  \eqref{volt-e06b}  is in $L_1$. 

In higher dimensions the situation is different: For $d\geq 2$ the operator $T$ is not any more continuous in $L_1$. In Theorem~\ref{volt-09} we show what can be obtained from equation \eqref{volt-e06a} or \eqref{volt-e06b}. 

Recall the definition of the weak-$L_p$ or Lorentz $L_{p,\infty}$ spaces, see for example \cite[(1.1.5), Def.~1.4.6]{Gr08}. Let $f:\rd\to\real$ be a Borel measurable function, define its decreasing rearrangement 
\begin{gather}\label{volt-e08}
	f^*(t) := \inf\left\{ s>0 \mid \Leb\left(\left\{|f(\cdot)|>s\right\}\right)\leq t\right\},
	\intertext{and the norm}\label{volt-e10}
	\|f\|_{p,\infty} := \sup_{t>0} t^{1/p} f^*(t), \quad 1\leq p \leq \infty. 
\end{gather}
The set of all $f$ with $\|f\|_{p,\infty}<\infty$ is denoted by $L_{p,\infty}$ and it is called \emph{Lorentz space} ($L_{p,\infty}$ coincides with the \emph{weak $L_p$-space}). Note that $L_1 \subset L_{1,\infty}$ and $L_\infty = L_{\infty,\infty}$. 
\begin{theorem}\label{volt-09}
	Let $d\geq 2$ and assume \textup{(A1)}--\textup{(A3)}.  Let $\rho\in L_1$ be the solution to \eqref{volt-e06b} in $\mathcal{S}'(\rd)$.
	\begin{enumerate}
		\item
		Suppose that $\rho\in W^{1,1}$. Then $ b\rho \in L_{1,\infty}$. 
		
		\item
		Suppose that $\int_{|y|\leq 1} |y|\, \nu(dy) < \infty$, then $b\rho\in L_{1,\infty}$. 
	\end{enumerate}	
\end{theorem}

\begin{remark}\label{volt-11}
	\emph{a.}\ \ 
	Theorem~\ref{volt-09}.a should be read as a \enquote{negative result}: If $b\rho \notin L_{1,\infty}$, then we cannot expect $\rho \in W^{1,1}$. 
	
	\medskip\noindent\emph{b.}\ \ 
	If (A3) and $\int_{|y|\leq 1} |y|\, \nu(dy) < \infty$ are satisfied, then it follows from the L\'evy--Khinchine representation \eqref{intro-e04} of $\psi$ that we have $|\psi(\xi)|\lesssim |\xi|$ for all $\xi \in \rd$. Thus, $|\widehat{\nabla K} (\xi)| = \frac{|\psi(\xi)|}{|\xi|} \lesssim 1$, $\xi\in \rd$. Using Plancherel's theorem, we see that $Tf = \nabla K*f$ is a bounded operator in $L_2$, and it follows from \cite[Chapter I\,\S5 Thm.~3]{St93} 
	that $T: L_p\cap L_2\to L_p$, $p\in (1,2]$, is continuous.
\end{remark}

\section{Proofs}\label{proo}

\noindent
We will now provide the proofs of the results mentioned in the previous section. We begin with several technical lemmas.

\begin{lemma}\label{proo-01}
	Let $d\geq 2$, $F\in C^{1,\alpha}(\rd,\rd)$ for some $\alpha\in (0,1]$ and $\diver F\in L_1  \cap L_\infty$.  If $d=2$ assume, in addition, that $\Lambda^\epsilon\diver F\in L_1$ for some $\epsilon>0$.  Then $\diver F  = \Delta\varphi$ for a potential $\varphi$,  which is bounded \textup{(}if $d>2$\textup{)} or grows at most like $\Lambda^\epsilon$ \textup{(}if $d=2$\textup{)}. 
\end{lemma}
\begin{proof} 
	Because of our assumptions, the Helmholtz decomposition for the vector field  $F$  yields that $\diver F = \diver  \nabla \varphi = \Delta \varphi$. Thus, $\varphi = P* (\diver  F)$, where $P$ is the Newton potential. The growth behaviour of $ \varphi$ follows directly from the estimate in the proof of Lemma~\ref{volt-04}.
\end{proof} 
\begin{remark}\label{proo-01b}
	With some effort one can use integration by parts and a few convergence arguments to show that $\varphi$ in Lemma~\ref{proo-01} is bounded, if we assume that $F\in C^{1,\alpha}(\rd)\cap L_1\cap L_\infty$ and $\diver F\in L_1$ ($d\geq 3$), resp., $\Lambda^\epsilon\diver F\in L_1$ ($d=2$). This gives then a slightly different set of conditions in Theorem~\ref{volt-05}.b. 
\end{remark}
We need the concept of a Calder\'on--Zygmund kernel, see e.g.\ \cite[I.8.18]{St93}. We say that a continuous function $Q:\rd\setminus\{0\}\to\real$ is a \emph{Calder\'on--Zygmund kernel}, if 
\begin{enumerate}
	\item $Q$ is homogeneous of degree $d$; in particular, $|Q(x)|\lesssim |x|^{-d}$ for all $x\neq 0$; 
	\item $\int_{|x|=1} Q(x)\,\sigma(dx)=0$, where $\sigma(dx)$ is the surface measure on $\Ss^{d-1}$; 
	\item $|Q(x+y) - Q(x)| \lesssim |y|$ for $|y|\ll 1$ and all $|x|=1$. 
\end{enumerate} 
The linear operator $\mathcal{Q} g := Q* g$ is called a \emph{Calder\'on--Zygmund operator}. It is known that $\mathcal{Q}$ is continuous from $L_2$ to $L_2$ and from $L_1$ to $L_{1,\infty}$, see \cite[I.8.18]{St93}. By the Marcinkiewicz interpolation theorem \cite[Thm.~I.4.5]{St70}, it is continuous from $L_p$ to $L_p$ for any $1<p<2$; see also \cite[I.8.18]{St93}. 

\begin{lemma}\label{proo-02}
	Let $P$ be the Newton potential for $d\geq 2$. Then $\nabla^2 P(x)$ is a matrix whose entries are Calder\'on--Zygmund kernels. 
\end{lemma}
\begin{proof}
	We have $\nabla^2 P(x) = \frac{1}{|x|^d} \left(\delta_{ik} - d\, \frac{x_i x_k}{|x|^2}\right)_{i,k=1}^d$. Therefore, the properties 1.~\& 3.~of a Calder\'on--Zygmund operator are straightforward.  Let us show that the entries $(\nabla^2 P)_{ik}$ of the matrix $\nabla^2 P$ satisfy property 2. Recall that $d\cdot\omega_d = \int_{|x|=1} \sigma(dx)$; since the measure $\sigma(dx)$ is rotationally symmetric, we have 
	\begin{gather*}
		d\int_{|x|=1} \frac{x_i x_k}{|x|^2}\, \sigma(dx) 
		= \delta_{ik} d\int_{|x|=1} \frac{x_i^2}{|x|^2}\,\sigma(dx)
		= \delta_{ik} \omega_d.  
	\end{gather*}
	This implies $\int_{|x|=1} (\nabla^2 P)_{ik}\, d\sigma = 0$. 
\end{proof}

\begin{lemma}\label{proo-03}
	Let $d\geq 3$ and $F$ be a vector field such that $F\in C^{1,\alpha}(\rd,\rd)\cap L_1\cap L_2$ for some $\alpha\in (0,1]$ and $\diver F\in L_1$. Then $\diver F = \Delta \varphi$ for some $\varphi$ such that $\nabla \varphi\in L_p$,  $1<p\leq 2$.
\end{lemma}

\begin{proof}
	Let $\varphi = P* (\diver  F)$, where $P$ is the Newton potential. Then $\Delta \varphi =\diver F$. Moreover, $\partial_j \varphi = \sum_{i=1}^d  \partial_i\partial_j  P* F_i$. Lemma~\ref{proo-02} shows that  $\partial_i\partial_j P$ are Calderon--Zygmund kernels. Since $F\in L_p \subset L_1\cap L_2$, $1<p\leq  2$, we can use \cite[p.~19, Thm.~3]{St93} and \cite[pp.~44-45, 8.18]{St93} to get $\nabla \varphi \in L_p$, $1<p\leq 2$.
\end{proof}

\begin{remark}\label{proo-04}
	Note that in Lemmas \ref{proo-01} and \ref{proo-03} we actually derived that  $F^{\mathsf{irr}} = \nabla \varphi$, and  $\varphi = P*\diver F$ is unique up to a constant. This follows from the  (generalized)  Liouville theorem, see \cite[Thm.~2.6,  Thm.~3.1]{Mi24} in $L^\infty$ resp.\ $L^p$.
\end{remark}

\begin{lemma}\label{proo-05}
	The operator $\widetilde{\psi}(D)\circ(-\Delta)^{-1}$ maps $ W^{1,1}\cap W^{1,\infty}\cap C^{1}(\rd)$ into $C_b(\rd)$. 
\end{lemma}	
\begin{proof}
	Let $d\geq 3$. Since $P(x) \asymp |x|^{2-d}$ and $|\nabla P(x)|\lesssim |x|^{1-d}$, we conclude that $P$, $\nabla P\in L_{1,\mathrm{loc}} \cap L_\infty(\rd\setminus B_1(0))$. For $f\in W^{1,1}\cap W^{1,\infty}$ we have $\nabla (P*f)= P*\nabla f$ and $\nabla^2(P*f)= (\nabla P)*(\nabla f)$ (to be read coordinatewise). We claim that $P*f$, $P*\nabla f$ and $(\nabla P)*(\nabla f)$ are bounded.
	
	Indeed, all expressions are of the form $Q*g$ where $Q\in L_{1,\mathrm{loc}} \cap L_\infty(\rd\setminus B_1(0))$ and $g\in L_1\cap L_\infty$. Using Young's inequality we get
	\begin{gather}\label{proo-e11}
		\|Q*g\|_{L_\infty} 
		\leq \|Q\One_{B_1(0)}\|_{L^1} \|g\|_{L^\infty} + \|Q\One_{\rd\setminus B_1(0)}\|_{L_\infty} \|g\|_{L_1}
		< \infty.
	\end{gather}
	and use for each term Young's inequality.
	
	Since $f$ is also in $C^1(\rd)$, we easily see that $P*f$, $P*\nabla f$ and  $(\nabla P)* (\nabla f)$  are  continuous. Thus, for $f\in C^1\cap W^{1,1}\cap W^{1,\infty}$ we have $P*f \in C_b^2(\rd)$. From the representation \ref{intro-e10} it is clear that $\widetilde{\psi}(D): C^2_b(\rd) \to C_b(\rd)$, proving the claim.  
	
	For $d=2$, the function $u:=P*f$ is unbounded even for $f\in C_c^\infty(\rd)$, but the derivatives $\nabla u  = (\nabla P)*f$ and  $\nabla^2 u  = (\nabla P)*(\nabla f)$ are bounded  for $f\in W^{1,1}\cap W^{1,\infty}$ by the same argument as in \eqref{proo-e11}.  Hence, using the representation of $\psi$ and (A3), we derive 
	\begin{gather*}
		\left\|\widetilde{\psi} (D) u \right\|_{L_\infty }
		\lesssim \left\|\nabla^2 u\right\|_{L_\infty} \int_{|u|\leq 1} |u|^2 \,\nu(du) 
		+ \left\|\nabla u\right\|_{L_\infty} \int_{|u|\geq 1} |u|\,\nu(du)
		<\infty. 
	\end{gather*}
	This completes the proof. 
\end{proof} 
Let us now analyze the structure of the right-hand side of \eqref{volt-e06a}.  Consider first the one-dimensional case. 

\begin{proposition}\label{proo-09}
	Let $d=1$. Assume \textup{(A1)}--\textup{(A3)}.   
	Then the  inverse  Fourier transform of $\frac 1{2\pi} \widetilde{\psi}(\xi)|\xi|^{-2}$ is given by
	\begin{gather*}
		\overline{\overline\nu}(x)
		:= \One_{(0,\infty)}(x) \int_x^\infty \bar\nu(y)\,dy - \One_{(-\infty,0)}(x)\int_{-\infty}^{x} \bar\nu(y)\,dy,
	\end{gather*}
	and $\frac d{dx}\overline{\overline\nu}(x) = -\overline{\nu}(x)$. Consequently,  $\nabla K = -\overline\nu$.
\end{proposition}
\begin{proof} 
	The definition of $K:=(2\pi)^{-d}\,\widetilde{\psi}(D)P$ is equivalent to $K*u = \widetilde{\psi}(D)(-\Delta)^{-1}u$, and the latter makes sense for all dimensions $d\geq 1$, see the proof of Theorem~\ref{volt-05} further down.  We are going to show that $K' = - \overline{\nu}$ if $d=1$. 
	
	Denote by $\nu_R(dy) := \One_{(0,R)}(|y|)\,\nu(dy)$ the truncated L\'evy measure and
	\begin{gather*}
		\psi_R(\xi) := \left(\int_0^R + \int_{-R}^0 \right) \left( 1- e^{i\xi y} + i\xi y \right) \nu(dy) = I_1 + I_2.
	\end{gather*}
	We show first that 
	\begin{gather}\notag
		i\xi \frac{\psi_R(\xi )}{\xi^2}= i\xi \,\widecheck{\overline{\overline{\nu}}_R}(\xi)
		= \widecheck{\overline{\nu}_R}(\xi), 
		\intertext{which becomes by changing $\xi \to -\xi$,}\label{proo-e12}
		\frac 1{2\pi}\,i\xi \frac{\widetilde{\psi}_R(\xi)}{\xi^2}= i\xi \,\widehat{\overline{\overline{\nu}}_R}(\xi)
		= -\widehat{\overline{\nu}_R}(\xi), 
	\end{gather}
	where
	\begin{align*}
		\overline{\nu}_R(x) 
		:=&\: \nu_R(x,\infty)\One_{(0,\infty)}(x) - \nu_R(-\infty,x)\One_{(-\infty,0)}(x)\\ 
		=&\: \nu(x,R)\One_{(0,R)}(x) - \nu(-R,x)\One_{(-R,0)}(x),
	\end{align*}
	so that $\frac d{dx} \overline{\overline\nu}_R(x) = -\overline{\nu}_R$. We consider only $I_1$, the calculations for $I_2$ are similar.  By Fubini's theorem we get
	\begin{align*}
		I_1 
		= \int_0^R \left(1- e^{i\xi y} + i\xi y\right) \nu(dy)
		&= \xi^2 \int_0^R\int_0^y\int_0^x e^{i\xi s}\,ds\,dx\,\nu(dy)\\
		&= \xi^2 \int_0^R\int_s^{R} \int_x^{R} \nu(dy)\,dx \,e^{i\xi s}\,ds\\
		&= \xi^2 \int_0^R e^{i\xi s}\,\overline{\overline{\nu}}_{R}(s)\,ds.
	\end{align*} 
	The use of Fubini's theorem is justified, since we can use Tonelli's theorem and the fact that $\nu$ is a L\'evy measure:
	\begin{gather*}
		\int_0^R \overline{\overline{\nu}}_R(s)\,ds
		= \int_0^R \int_s^R \int_x^R \nu(dy)\,dx\,ds
		= \int_0^R \int_0^y \int_0^s\,dx\,ds\,\nu(dy)
		= \frac 12 \int_0^R y^2\,\nu(dy).
	\end{gather*}
	The last integral being finite, $I_1$ is well-defined and it is the inverse Fourier transform of $\overline{\overline{\nu}}_R$ on the positive half-axis. 
	The second equality in \eqref{proo-e12} follows  by integration by parts.

	Note that $a_R(\xi):=i\xi \widetilde{\psi}_{R}(\xi)\xi^{-2}$ belongs to $\mathcal{S}'(\real)$ and converges pointwise to $a(\xi):=i\xi \widetilde{\psi}(\xi)\xi^{-2}$ as $R\to \infty$; since  both functions $a_R(\xi)$ and $a(\xi)$ are bounded from above by  a constant multiple of $(1+|\xi|)$ (cf.\ Remark~\ref{volt-03}.b), we also have a convergence in $\mathcal{S}'(\real)$. Thus,   $- 2\pi \,\overline{\nu}$ is the inverse Fourier transform of $a(\xi)$. Since $a(D)f = (a\widehat f){\widecheck{\phantom{I}}} =  (2\pi)^{-1}\widecheck a* f$, we obtain 
	\begin{equation*}
		\nabla\widetilde{\psi}(D)(-\Delta)^{-1} f(x)  = -\overline{\nu} * f(x),\quad f\in\mathcal{S}(\real). 
		\qedhere
	\end{equation*}
\end{proof}
We can now prove Theorem~\ref{volt-05}.
\begin{proof}[of Theorem~\ref{volt-05}]
	By assumption, $b_\ell$ is measurable and grows at most like a polynomial. Therefore, $\diver(b_\ell \rho)$ is for $\rho\in L_1$ an element of $\mathcal{S}'(\rd)$. Moreover, the function $i\xi\widetilde{\psi}(\xi)|\xi|^{-2}$ can be bounded by $C\max\{1,|\xi|\}$, see Remark~\ref{volt-03}.b., and for $\rho\in L_1$ we conclude that $i\xi\widetilde{\psi}(\xi)|\xi|^{-2}\widehat\rho(\xi)$ defines an element of $\mathcal{S}'(\rd)$, hence the Fourier transform $-\nabla\widetilde{\psi}(D)(-\Delta)^{-1}\rho\in \mathcal{S}'(\rd)$.  Since the (formal) adjoint of $\psi(D)$ is $\widetilde\psi(D)$,  this shows that $L^*\rho = 0$ in $\mathcal{S}'(\rd)$ makes sense and that it is equivalent to \eqref{volt-e06}.
	
	\medskip\noindent
	\emph{Proof of}~\ref{volt-05}.a. 
	Let $d>1$.
	Since $a(D)u = (2\pi)^{-d}\,\widecheck a*u$, we see that $\nabla \widetilde{\psi}(D)(-\Delta)^{-1}\rho = \nabla K*\rho$ if we use $K(x):=(2\pi)^{-d}\,\widetilde{\psi}(D)P(x)$. The case $d=1$ has already dealt with in Proposition~\ref{proo-09}.
	Thus, any of the conditions \eqref{volt-e06a} or \eqref{volt-e06b} imply \eqref{volt-e06}, hence $L^*\rho =0$.

	\medskip\noindent\emph{Proof of}~\ref{volt-05}.b. 
	\textbf{i)\ \ } In one dimension $\diver = \frac d{dx}$ and we get immediately from \eqref{volt-e06} that $\nabla K*\rho = b_\ell\rho + \text{const}$. 
	
	\medskip\noindent
	\textbf{ii), iii)\ \ } By assumption, we have  $\rho, b\rho, b_\ell\rho \in L_{1}$. We already know from \eqref{volt-e06} that $L^*\rho = 0$ can be expressed as
	\begin{gather*}
		\diver\left( b_\ell\rho - \nabla \circ \widetilde{\psi}(D) \circ (-\Delta)^{-1}\rho \right) =0
		\quad\text{in}\quad \mathcal{S}'(\rd).
	\end{gather*}
	Using that $\diver (b_\ell\rho) = \diver \nabla \varphi = \Delta \varphi$, where $\varphi = P * \diver (b_\ell\rho)$, we get
	\begin{gather*}
		\diver\left(\nabla \varphi - \nabla K * \rho\right)
		= \Delta (\varphi - K* \rho ) = 0 
		\quad\text{in}\quad \mathcal{S}'(\rd).
	\end{gather*}
	Lemma~\ref{proo-01} shows that $\varphi$ is bounded  or satisfies $\lim_{|x|\to\infty} \frac{ \varphi(x)}{|x|}=0$,  and Lemma~\ref{proo-05} proves that $K* \rho$ is bounded.  Thus, the generalized classical Liouville theorem,  see e.g.\ \cite[Thm.~3.1]{Mi24}, can be applied and shows that $\varphi - K* \rho = \mathrm{const.}$, implying \eqref{volt-e06a}. 
	
	\medskip\noindent
	\textbf{iv)\ \ }
	Our assumptions show that $\nabla \varphi, \nabla K*\rho\in L_p$, $1<p\leq 2$, see Lemma~\ref{proo-03}, Lemma~\ref{proo-02} and the paragraph preceding that lemma. We conclude from the $L_p$-version of the Liouville theorem (cf.~\cite[Thm.~2.6]{Mi24}) that $\varphi - K*\rho = \text{const.}$, hence, 
	\begin{gather*}
		(b_{\ell}\rho)^{\mathsf{irr}}= \nabla \varphi =\nabla  K* \rho.   \qedhere 
	\end{gather*}
\end{proof}

\begin{remark}\label{proo-13}
	If $\psi(\xi)= |\xi|^\alpha$, $\alpha\in (0,2)$, then 
	$K(x) 
	= \mathscr{F}^{-1}_{\xi\to x}\left[|\xi|^{-(2-\alpha)}\right]$ is the Riesz kernel, cf.\ \cite[pp.~112--113]{Gr09}, i.e.\ the tempered distribution given by the kernel 
	\begin{gather}\label{exa-e16}
		K(x) = 
		\pi^{-d/2} 2^{\alpha-2}\frac{\Gamma\left(\frac{d+\alpha}{2}-1\right)}{\Gamma\left(\frac{2-\alpha}{2}\right)}\, |x|^{2-\alpha-d},
		\quad  0<\alpha<2.
	\end{gather}
\end{remark}
We will now consider dimensions $d\geq 2$. Assume (A3). Then we have 
\begin{align*}
	\left(\frac{\widetilde{\psi}(\cdot)}{|\cdot|^2}\right)\!\!\widecheck{\phantom{\frac{M}{I}}}\!\!\!\!\!\!(x)
	= C_d  \widetilde{\psi}(D)|x|^{2-d} 
	&=  C_d  \int_\rd \left[\frac{1}{|x-y|^{d-2}}- \frac{1}{|x|^{d-2}} + \frac{(2-d) x\cdot y}{|x|^d }\right]\nu(dy) 
	\intertext{or, if $\nu$ is symmetric,} 
	&=  C_d \, \mathrm{p.v.}\!\!\int_\rd \left(\frac{1}{|x-y|^{d-2}}-\frac{1}{|x|^{d-2}}\right)\nu(dy).  
\end{align*}

\begin{remark}\label{proo-15}
	In dimension $d=1$, the $\overline{\nu}$ is, in general, not of Calder\'on--Zygmund type, since it lacks homogeneity. Nevertheless, we can use methods of harmonic analysis to ensure that $Tf := \overline{\nu}*f$ is bounded from $L_1$ to $L_{1,\infty}$. The key is the so-called \emph{Hörmander condition}  (see \cite[p.~25 (18)]{St93}): For any $c>1$ there is a constant $A$ such that for all $y\neq 0$ 
	\begin{gather*} 
		\smash[t]{\int_{|x|\geq c|y|} \left|\overline{\nu}(x-y)-\overline{\nu}(x)\right| dx \leq A,}
	\end{gather*}
	which holds, if $\nu$ has the integrated tail property  $\int_{x\neq 0}|y|\,\nu(dy) = \int_{x\neq 0} |\overline\nu(x)|\,dx < \infty$, see the discussion before Proposition~\ref{volt-07}.  The claim now follows from the proof of \cite[Theorem 3, pp.\ 19--21]{St93}. The integrated tail property also entails that $\widecheck{\overline{\nu}}(\xi) = \frac{\psi(\xi)}{i\xi}$ is bounded (use \eqref{proo-e12}, Remark~\ref{volt-03}.b.~and a similar estimate for large $|\xi|$); thus, hence $T$ maps $L_2$ into $L_2$ and, by interpolation, $L_p$ into $L_p$ for all $p\in (1,2)$.
\end{remark}

\begin{proof}[of Theorem~\ref{volt-09}]a)\ \ 
	Set $\mathcal{K} := \nabla K$ and note that 	
	\begin{gather*}
		\Phi(x,y) 
		:= P(x - y) - P(x) + \nabla P(x) \cdot y
		= \int_0^1 (1-t) y\cdot \nabla^2_x P(x-ty) y\,dt,
	\end{gather*}
	where $P(x)$ is the Newton kernel. Split 
	\begin{align*}
		\mathcal{K} (x) 
		= \int_\rd  \nabla_x \Phi(x,y)\, \nu(dy) 
		&= \left(\int_{|y|\leq 1} + \int_{|y|>1}\right) \nabla_x \Phi(x,y)\, \nu(dy) \\
		&=: \mathcal{K}_{\textrm{small}}(x) + \mathcal{K}_{\textrm{big}}(x). 
	\end{align*}
	If $\rho\in W^{1,1}$, then $\mathcal{K}_{\textrm{small}} *\rho = K_{\textrm{small}}* \nabla \rho$  (the convolution is meant coordinatewise), where $K_{\textrm{small}}(x,y) := \int_{|y|\leq 1} \Phi(x,y)\, \nu(dy)$, implying that
	\begin{align*}
		\mathcal{K}_{\textrm{small}} * \rho (x) 
		&= \int_{\rd} \int_{|y|\leq 1} \int_0^1 (1-t) y\cdot \nabla^2 P(z - ty)y \nabla \rho(x-z)\,dt\,\nu(dy)\,dz\\
		&= \int_{|u|\leq 1} \int_0^1  (1-t) \left(y\cdot\nabla^2 P(\cdot) y * \nabla \rho\right)(x - t y)\,  dt\,\nu(dy). 
	\end{align*}
	By Lemma~\ref{proo-02},  $Q:=\nabla^2 P$ is a  matrix of Calder\'on--Zygmund kernels,  hence, the operator $\mathcal{Q} f := Q*f$  (defined coordinatewise)  maps $L_1$ continuously into $L_{1,\infty}$. Now we use the Young inequality in Lorentz spaces, cf.\ \cite[Thm.~2.6]{ON63} or \cite{Bl73}: If $m(\cdot)$ is a finite measure on $\rd$, then 
	\begin{gather}\label{proo-e16}
		\|\mathcal{Q} f * m \|_{1,\infty} \lesssim \|\mathcal{Q} f \|_{1,\infty} \lesssim \|f\|_1, \quad f\in L_1. 
	\end{gather}
	Using this estimate for the operator $\mathcal{K}_{\textrm{small}}$ and the finite measure $|y|^2 \One_{(0,1]}(|y|)\,\nu(dy)$, we arrive at 
	\begin{gather}\label{proo-e18}
		\|\mathcal{K}_{\textrm{small}} *\rho \|_{1,\infty} \lesssim  \|\nabla \rho \|_1. 
	\end{gather}
	Consider now $\mathcal{K}_{\textrm{big}}$. By (A3), 
	\begin{align*}
		\mathcal{K}_{\textrm{big}}* \rho (x) 
		&= \int_\rd \int_{|y|>1} \int_0^1  \left(\nabla^2 P(z) - \nabla^2 P(z - ty)\right) y \rho(x-z)\,dt \,\nu(dy) \,dz \\
		&= \int_{|y|>1} \int_0^1  \left(\mathcal{Q}\rho(x) - \mathcal{Q}\rho(x -  t y)\right)y  \,dt\,\nu(du). 
	\end{align*}
	Using again \eqref{proo-e16} (for the finite measure $\One_{(1,\infty)}(|y|)\,\nu(dy)$) and (A3), we arrive at
	\begin{gather}\label{proo-e20}
		\|\mathcal{K}_\textrm{big}* \rho\|_{1,\infty}\lesssim \|\rho\|_{1}. 
	\end{gather}
	Thus, a)  follows from \eqref{proo-e18} and \eqref{proo-e20}. 
	
	\bigskip\noindent
	b) If we assume  (A3) and  $\int_{|y|\leq 1} |y|\, \nu(dy) < \infty$, then we may write 
	\begin{align*}
		\mathcal{K} * \rho(x) 
		=  \int_\rd \int_0^1  \left( \mathcal{Q}\rho(x) - \mathcal{Q}\rho(x - t y) \right) y  \,dt\,\nu(dy), 
	\end{align*}
	and b)  follows by the same argument as for $\mathcal{K}_\textrm{big}$ in the first part of the proof. 
\end{proof}

\section{L\'evy-driven Ornstein--Uhlenbeck process}\label{exa}

\noindent
Consider the special case of a L\'evy-driven Ornstein--Uhlenbeck process in $\rd$: 
\begin{gather}\label{exa-e02}
	dX_t = - QX_t dt + d Z_t, \quad t>0, 
\end{gather}
where $Q$ is a $d\times d$ matrix whose eigenvalues have strictly positive real parts (notation: $Q\in M_+(\rd)$), and  $(Z_t)_{t\geq 0}$ is a L\'evy process in $\rd$ with characteristic function $\psi(\xi)$, $\xi \in \rd$.  It is shown in \cite[Thm.~4.1, Thm.~4.2]{SY84} that $Q\in M_+(\rd)$ together with the moment condition 
\begin{gather}\label{exa-e04}
	\int_{|y|\geq \epsilon} \log^+|y|\,\nu(dy) < \infty
\end{gather}
implies $X_t\xRightarrow{\text{law}} X_\infty$, and that the limiting distribution is a self-decomposable probability measure.  This result was first proved in \cite[Thm.~1]{Wo82} in the one-dimensional case; see also \cite{ARJ00}, where the invariant  were described by analytic methods. If \eqref{exa-e04} fails, then the process has no  invariant measure.  Moreover, the Fourier transform of the stationary density is given by 
\begin{gather*}
	\widehat{\rho}(\xi)   
	= (2\pi)^{-d}\exp\left[{-\int_0^\infty \psi\left(e^{-s Q^\top}\xi\right) ds}\right]
	= (2\pi)^{-d}\exp\left[{-\int_0^1 \psi\left(t^{Q^\top}\xi\right)\,\frac{dt}{t}}\right],  
\end{gather*}
where we use the notation $t^Q:= e^{(\log t)Q}$ and the substitution $s=-\log t$. 

Let us show how this result can be  recovered  by our approach. We have 
\begin{gather*}
	\diver \left(Q x \rho - \nabla \circ \psi(D)\circ (-\Delta)^{-1} \rho  \right) =0.
\end{gather*}
By the linearity of the Fourier transform, $\widehat{(Q x \rho)}(\xi) = iQ \nabla  \widehat{\rho}(\xi)$. Then \eqref{volt-e06} becomes 
\begin{gather*}
	\smash[t]{Q\nabla \widehat{\rho}(\xi) 
	= -\widehat\rho(\xi) \frac{\psi(\xi)}{|\xi|^2} \xi \quad \text{a.e.}}
\end{gather*}
Scalar multiplication with the vector $\xi$ shows 
\begin{gather*}
	Q\nabla \widehat{\rho}(\xi) \cdot \xi = - \widehat\rho(\xi) \psi(\xi),  
\end{gather*}
which we can also write as 
\begin{gather*}
	\nabla \log \widehat{\rho}(\xi) \cdot Q^\top\xi = - \psi(\xi). 
\end{gather*}
In order to solve this equation we observe that 
\begin{gather*}
	\log \widehat{\rho}(0) - \log \widehat{\rho}(\xi)
	= \int_0^\infty \frac{d}{ds} \log \widehat{\rho}\left(e^{-s Q^\top}\xi\right) ds.   
\end{gather*}
Therefore,
\begin{align*}
	\log \widehat{\rho}(0) - \log \widehat{\rho}(\xi)
	&= -\int_0^\infty \left(\nabla \log \widehat{\rho}\right)\left(e^{-s Q^\top}\xi\right) Q^\top e^{-s Q^\top}\xi \, ds \\
	& = \int_0^\infty \psi\left(e^{-sQ^\top}\xi\right) ds,  
\end{align*}
where we use $\frac{d}{dt} e^{- t A} \xi = - A e^{-tA} \xi$. As $\widehat\rho(0)=(2\pi)^{-d}$, we get
\begin{gather}\label{exa-e08}
	\widehat{\rho}(\xi) 
	= (2\pi)^{-d}\exp\left[{-\int_0^\infty \psi\left(e^{-s Q^\top}\xi\right) ds}\right]. 
\end{gather}

\begin{remark}\label{exa-02}
	As a corollary, we also recover the result from \cite[Lem.~5.1]{BN98}, see also \cite{Ma04}, that the function
	\begin{gather*}
		-\nabla  \log\widehat\rho(\xi) \cdot Q^\top \xi,
	\end{gather*}
	is negative definite (in the sense of Schoenberg)  if $Q\in M_+(\rd)$.
\end{remark}

\begin{remark}\label{exa-07}
	Let $(X_t)_{t\geq 0}$ be a Feller process with generator $L= b(x) \cdot \nabla - \psi(D)$ and domain $D(L)$. Assume that $\sigma(\cdot) \in C_b(\rd,\real)$ is strictly positive and set
	\begin{gather*}
		A_t := \int_0^t \frac{ds}{\sigma(X_s)}, \quad \tau(t):= \inf\{s > 0 \mid A_t >s\}. 
	\end{gather*}
	By \cite[Cor.~4.2]{BSW13}  
	the time-changed process $Y_t:= X_{\tau(t)}$ is again a Feller process whose generator is the closure of $(\sigma(\cdot)L, D(L))$. Thus, if the process $X$ has a unique invariant probability measure with density $\rho(x)$, then $\rho^\sigma(x) := C\rho(x)/\sigma(x)$ is the unique invariant probability measure for $Y$; here $C>0$ is the normalization constant. 
\end{remark}

\section{L\'evy-driven SDE vs.\ reflecting  Brownian motion}\label{app}

\noindent
The structure results on the invariant measure $\rho$ allow us to construct SDEs which can be used to approximate (functionals of) $\rho$ by simulation methods.  The paper \cite{Ol24} considers this problem for $d=1$,  and contains an explicit construction of a L\'evy-driven SDE, which may be used for simulation of (functionals of) $X_t$. If we also allow models with reflection, we can have several SDEs whose solutions have the same invariant measure. This is illustrated in the following example, and it is a natural question to ask which model is more efficient. This depends heavily on the structure of the invariant measure. Intuitively, if the invariant measure is concentrated in a small region $D$ (\enquote{thin tails}), it can be  accessed faster by a continuous motion rather than by a jump SDE, since jumps will drive the process quickly away from  $D$. Therefore, if $\rho$ has heavy tails, the jump model should be more efficient. This idea is confirmed in the example below, where we take $\rho(x)= \lambda e^{-\lambda x}$, $x>0$, and construct a L\'evy-driven SDE and a reflecting  Brownian motion, both having $\rho$ as the invariant measure. In both cases the solutions to the SDEs and the corresponding invariant measures are unique: Indeed, in the case of the L\'evy-driven SDE it is possible to pick a L\'evy measure satisfying \eqref{exa-e04}; then by \cite[Thm.~4.1, Thm.~4.2]{SY84} we have a unique invariant distribution, given by \eqref{volt-e06}. In the case of the  reflecting Brownian motion, we have a strictly positive transition probability density (cf.~\eqref{app-e04} below), i.e.\ we have irreducibility and the strong Feller property, hence a unique invariant measure, cf.\ Remark~\ref{volt-03}. Since we have for bigger $\lambda$ a faster decay of $\rho(x)$ as $x\to\infty$, one would expect that for small $\lambda$ the L\'evy-driven SDE will outperform the model given by the reflecting Brownian motion. To prove this guess, we will now compare the speed of convergence for the first moment of the processes.

\bigskip\noindent\textbf{L\'evy-driven SDE.}
Consider the L\'evy-driven SDE \eqref{intro-e14}, where $Z_t$ is a pure-jump L\'evy process with L\'evy measure $\nu(dy) =\rho(y)\,dy$, and $b(x) =-x+\ell$. Then $\overline{\nu}(x)= e^{-\lambda x}$, $x>0$, and \eqref{volt-e06b} holds with $\tilde{b}(x) := b(x) - \ell = -x$. One can also calculate explicitly $\psi(\xi)$, but in this setting it is easier to start directly with the L\'evy measure. Note that in this case $\pi(dx) = \rho(x)\,dx$ is also the invariant measure for the solution to \eqref{intro-e14}. We can solve \eqref{intro-e14} explicitly: 
\begin{gather*}
	X_t= x e^{-t} + \int_0^t e^{-(t-s)}\,dZ_s  
	\quad\text{and}\quad
	X_\infty = \int_0^\infty e^{-s}\,dZ_s\sim \pi. 
\end{gather*}
The moments of $X_t$ can be calculated directly, namely 
\begin{gather}\label{app-e02}
	\Ee X_t - \Ee X_\infty
	= e^{-t}\left(x-\frac{1}{\lambda}\right). 
\end{gather}
In this case, up to a multiplicative constant, the rate of convergence is independent of $\lambda$. 

\bigskip\noindent\textbf{Reflecting Brownian motion.}
Consider now a Langevin diffusion with a  reflecting  Brownian motion. In this case, one may find the expected value of the process $X_t$ via the density function of the process itself, which is given by (cf. \cite{GW18})
\begin{small}
\begin{gather}\label{app-e04}
	p(t,x,y) 
	= \lambda e^{-y\lambda}\Phi\left(\frac{\frac{\lambda t}{2}-x-y}{\sqrt{t}}\right) 
	+ \frac{1}{\sqrt{t}}\varphi\left(\frac{-\frac{\lambda t}{2}+x-y}{\sqrt{t}}\right)
	+ \frac{e^{-y\lambda}}{\sqrt{t}}\varphi\left(\frac{\frac{\lambda t}{2}-x-y}{\sqrt{t}}\right),
\end{gather}\end{small}%
where $\Phi(\cdot)$ and $\varphi(\cdot)$ denote the cumulative distribution function (CDF) and the probability density function of a standard normal random variable, respectively. Therefore, \begin{small}%
\begin{align*}
	\Ee X_t 
	&= \int\limits_0^{\infty} \! y\left[\lambda e^{-y\lambda}\Phi\left(\frac{\frac{\lambda t}{2}-x-y}{\sqrt{t}}\right)+\frac{1}{\sqrt{t}}\varphi\left(\frac{-\frac{\lambda t}{2}+x-y}{\sqrt{t}}\right)+\frac{e^{-y\lambda}}{\sqrt{t}}\varphi\left(\frac{\frac{\lambda t}{2}-x-y}{\sqrt{t}}\right)\right] dy. 
\end{align*}\end{small}%
Integrating by parts and collecting the terms, we derive\begin{small}%
\begin{gather*}
	\Ee X_t 
	= \frac{1}{\lambda} 
	+ \left(x-\frac{1}{\lambda}-\frac{\lambda t}{2}\right) \Phi\left(\frac{x-\frac{\lambda t}{2}}{\sqrt{t}}\right) 
	- \frac{1}{\lambda}e^{x\lambda}\Phi\left(\frac{-\frac{\lambda t}{2}-x}{\sqrt{t}}\right) 
	+ \sqrt{t}\varphi\left(\frac{-\frac{\lambda t}{2}+x}{\sqrt{t}}\right).
\end{gather*}\end{small}%
Set  $-z_{t} := \frac{1}{\sqrt{t}}\left(x-\frac{\lambda t}{2}\right)$. Using Mill's ratio, we obtain the following asymptotics as $t\to \infty$:
\begin{gather*}
	\Ee X_t - \Ee X_\infty  
	=
	\frac{16\varphi(z_t)}{\lambda^3}
	\left[
	\left(\frac{x}{\lambda}-2\right)
	t^{-3/2}+O\left(t^{-5/2}\right)
	\right]
\end{gather*}
or, equivalently, 
\begin{gather}\label{app-e16}
	\Ee X_t - \Ee X_\infty  =\frac{16e^{\frac{\lambda x}{2}} e^{-\frac{\lambda^2 t}{8}} }{\lambda^3\sqrt{2 \pi}} \left[
	\left(\frac{x}{\lambda}-2\right)
	t^{-3/2}+O\left(t^{-5/2}\right)
	\right].
\end{gather}
Observe that in this case the convergence rate depends on $\lambda$; theoretically, the Lévy-driven model \eqref{intro-e14} outperforms the  reflecting  BM model \eqref{intro-e26} for $\lambda < \sqrt{8}$, and vice versa. To illustrate this, denote
\begin{gather*}
	\epsilon(t) := \Ee X_t - \Ee X_\infty, \quad x=2.5.
\end{gather*}
Figure \ref{app-fig} displays the deviations $\epsilon(t)$ from the theoretical values for both models under different values of $\lambda$. The vertical lines indicate the time moments from which onwards each $|\epsilon(t)|$ does not exceed $\delta = 0.01$. Note that the reflecting BM model works significantly worse for small $\lambda$. 
\begin{figure}[!ht]
	\mbox{}\hfill
	\includegraphics[width=.49\textwidth]{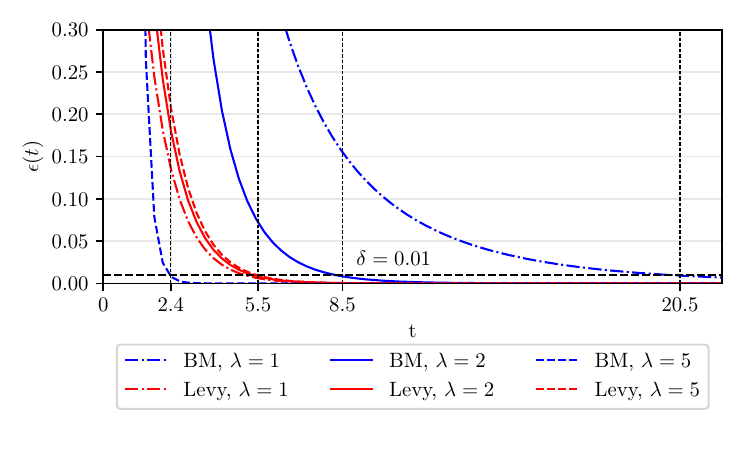}
	\hfill
	\includegraphics[width=.49\textwidth]{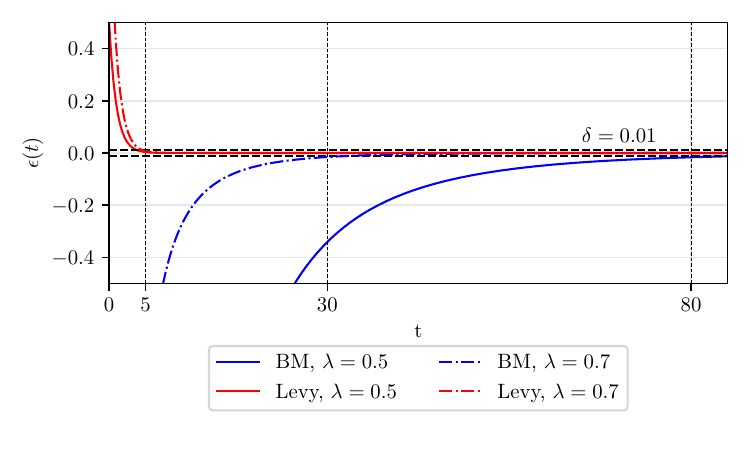}
	\hfill\mbox{}
	\caption{Both graphs illustrate how fast different models (reflecting BM and L\'evy-SDE) converge. The function $\epsilon(t)$ is the difference between the estimated and theoretical mean values.  
	The vertical lines indicate the moment of convergence for each model for different values of parameter $\lambda$. 
	}\label{app-fig}
\end{figure}

\begin{remark}\label{app-07} 
	The relations \eqref{app-e02} and \eqref{app-e16} can be generalized for Lipschitz functionals $\varphi(X_t)$  of $X_t$, and we get the same order for the upper bounds. In particular, we get an upper estimate for the Wasserstein $\mathds{W}_1$-distance between $\Law(X_t)$ and $\Law(X_\infty)$: 
	\begin{gather}
		\mathds{W}_1\left(\Law(X_t), \Law(X_\infty)\right)
		:= \sup_{\varphi\in \mathrm{Lip}_1}\left| \int \varphi\,dP_t - \int \varphi\,dP_\infty \right|, 
	\end{gather}
	where $\mathrm{Lip}_1$ is the set of Lipschitz continuous functions with Lipschitz constant $1$: $|\varphi(x)- \varphi(y)|\leq |x-y|$.   This allows us to conclude which type of approximation of $X_\infty$ is preferable. 
\end{remark} 

\begin{acknowledgement}
	We thank David Berger (TU Dresden) for interesting and helpful discussions. The two first-named authors are gratefully acknowledging financial support through the \emph{DAAD Ostpartnerschaften} programme managed by TU Dresden. 
\end{acknowledgement}

\end{document}